\documentclass[12]{article}
\usepackage{amsmath,amsthm,amsfonts}

\newtheorem{theorem}{Theorem}[section]
\newtheorem{lemma}[theorem]{Lemma}
\newtheorem{corollary}[theorem]{Corollary}
\newtheorem{example}[theorem]{Example}

\newcommand{\LL}{{\cal L}}
\newcommand{\T}{{\cal T}}

\newcommand{\F}[1]{{\cal C}(#1)}

\DeclareMathOperator{\PG}{PG}

\date{}

\title{The chromatic number of a two families of generalized Kneser graphs related to finite generalized quadrangles and finite projective $3$-spaces}
\author{Klaus Metsch\thanks{Justus-Liebig-Universit\"{a}t, Mathematisches Institut,
Arndtstra{\ss}e 2, D-35392 Gie{\ss}en}}

\begin{document}
\maketitle

\begin{abstract}
Let $\Gamma$ be the graph whose vertices are the chambers of the finite projective space $\PG(3,q)$ with two vertices being adjacent when the corresponding chambers are in general position. It is known that the independence number of this graph is $(q^2+q+1)(q+1)^2$. For $q\ge 43$ we determine the largest independent set of $\Gamma$ and show that every maximal independent set that is not a largest one has at most constant times $q^3$ elements. For $q\ge 47$, this information is then used to show that $\Gamma$ has chromatic number $q^2+q$. Furthermore, for many families of generalized quadrangles we prove similar results for the graph that is built in the same way on the chambers of the generalized quadrangle.
\end{abstract}



\section{Introduction}

A \emph{chamber} of a projective $3$-space is a set consisting of a point $P$, a line $\ell$ and a plane $\pi$ that are mutually incident. Two chambers $\{P,\ell,\pi\}$ and $\{P',\ell',\pi'\}$ are called \emph{opposite}, if $P$ does not lie in $\pi'$, if $\ell$ is skew to $\ell'$ and if $P'$ does not lie in $\pi$. We are interested in the graph whose vertices are chambers of a three dimensional projective space of finite order $q$ where two vertices are adjacent, if the corresponding chambers are opposite. It was shown in \cite{Sam} that the independence number of this graph is $(q^2+q+1)(q+1)^2$. For each point or plane $x$ the set of all chambers whose line is incident with $x$ is an independent set of this size. We denote this independent set by $\F{x}$.

A \emph{chamber} of a generalized quadrangle is a set consisting a point $P$ and a line $\ell$ that are incident. Two chambers $\{P,\ell\}$ and $\{P',\ell'\}$ are called \emph{opposite}, if $\ell$ and $\ell'$ have no common point and if there is no line containing $P$ and $P'$. A generalized quadrangle of order $(s,t)$ is called \emph{thick}, if $s,t\ge 2$. Again, we are interested in the graph whose vertices are chambers of a finite thick generalized quadrangle where two vertices are adjacent, if the corresponding chambers are opposite. We will determine the largest independent sets of these graphs and for many families of finite thick generalized quadrangles we also determine the chromatic number of this graph.

More generally one can consider flags of a certain type $S\subseteq\{1,2,\dots,n\}$ of a finite building of rank $n$ and define the graph whose vertices are these flags with two vertices being adjacent when the corresponding flags are in general position. For many buildings and many types $S$ , neither the independence number nor the chromatic number of these graphs are known. Here we focus on results on projective spaces or, equivalently, vector spaces. For $|S|=1$ one obtains the so called $q$-Kneser graphs and for these there are results on the independence number and the chromatic number, see \cite{Frankl&Wilson} and \cite{chromaticqKneser,BlokhuisHIltonMilner}. For $|S|>1$, the chromatic number is known only in a few cases, see \cite{lineplaneflags,Jozefien} and the references therein. For many sets $S$ with $|S|>1$, the independence number was determined in \cite{Sam} but neither the chromatic number nor the structure of the largest independent sets is known. This also applies to the graph defined above whose vertices are chambers of a projective 3-space of order $q$.

The independence number $(q^2+q+1)(q+1)^2$ of this graph was obtained using algebraic considerations. Here we use a geometric approach. This has the advantage, as is quite typical,  that we do not only determine the independence number but also classify the largest independent sets and moreover obtain a Hilton-Milner type result (a bound for the second largest maximal independent sets). The drawback is that our proof only gives interesting results for $q\ge 43$. One consequence of the Hilton-Milner type result in (b) of our first theorem below is that it enables us to determine the chromatic number of the graph in part (c).

\begin{theorem}\label{mainPG}
Let $\Gamma$ be the graph whose vertices are the chambers of a projective $3$-space of finite order $q$ with two vertices being adjacent when the corresponding chambers are in general position. Suppose that $q\ge 43$. Then we have.
\begin{enumerate}
\renewcommand{\labelenumi}{\rm (\alph{enumi})}
\item The largest independent sets of $\Gamma$ are the sets $\F{x}$ for points and planes $x$.
\item Every maximal independent set that is not of the form $\F{x}$ for a point or plane $x$ has at most $9(q+1)(5q^2+1)$ elements.
\item If $q\ge 47$, then the chromatic number of $\Gamma$ is $q^2+q$.
\end{enumerate}
\end{theorem}

Our second result is on finite generalized quadrangles \cite{Payne&Thas}. An \emph{ovoid} of a generalized quadrangle is a subset of its point set that meets every line in a unique point, and, dually, a \emph{spread} of a generalized quadrangle is a subset of its line set that partitions the point set. For a generalized quadrangle of finite order $(s,t)$ we use the following notation. If $P$ is a point, then $\F{P}$ denotes the set of all chambers whose line is incident with $P$. Dually, if $\ell$ is a line, then $\F{\ell}$ denotes the set of all chambers whose point is incident with $\ell$. Clearly, these are independent sets of cardinality $(s+1)(t+1)$.

We remark that there is exactly one generalized quadrangle of order $(2,2)$. Its points and lines are the points and lines of the (unique) non-degenerate quadric of the 4-dimensional projective space $\PG(4,2)$ of order $2$.

\begin{example}\label{examp}
Consider the graph $\Gamma$ obtained from generalized quadrangle of order $(2,2)$ embedded in $\PG(4,2)$ and consisting of the points and lines of a parabolic quadric $Q(4,2)$ of $\PG(4,2)$. Consider a hyperbolic quadric $Q^+(3,2)$ contained in $Q(4,2)$ obtained by intersecting $Q(4,2)$ with a suitable $3$-space of $\PG(4,2)$. Then the set $X$ consisting of the nine chambers $\{P,\ell\}$ with $P$ a point of $Q^+(3,2)$ and $\ell$ the unique line of $Q(4,2)$ on $P$ that is not contained in $Q^+(3,2)$ is an independent set of $\Gamma$, since for any two such chambers $\{P,\ell\}$ and $\{P',\ell'\}$ either $P$ and $P'$ lie on a line of $Q^+(3,2)$ or otherwise $\ell$ is the unique line of $Q(4,2)$ on $P$ that meets $\ell'$.
\end{example}

\begin{theorem}\label{mainGQ}
Let $\Gamma$ be the graph whose vertices are the chambers of a thick generalized quadrangle of finite order $(s,t)$ with two vertices being adjacent if the corresponding chambers are in general position. Then we have.
\begin{enumerate}
\renewcommand{\labelenumi}{\rm (\alph{enumi})}
\item The independence number of $\Gamma$ is $\alpha(\Gamma)=(s+1)(t+1)$.
\item Then independent sets of $\Gamma$ of cardinality $\alpha(\Gamma)$ are the sets $\F{x}$ for a point or line $x$ and, if $(s,t)=(2,2)$, the independent sets constructed in Example \ref{examp}.
\item Every maximal independent set with less than $(s+1)(t+1)$ elements has at most $\max\{1+s+2t,1+t+2s\}$ elements.
\item If the generalized quadrangle has an ovoid or a spread, then $\Gamma$ has chromatic number $st+1$, otherwise it has chromatic number at least $st+2$.
\end{enumerate}
\end{theorem}

Several families of classical generalized quadrangles have an ovoid or a spread and hence the chromatic numbers of the corresponding graphs are known.

\textsc{Remark}. The independent sets of graphs arising from flags in buildings, especially in projective spaces, are often called Erd\H os-Ko-Rado sets for historical reason. Information can be obtained from our reference list.

\section{Generalized quadrangles}

In this section we prove Theorem \ref{mainGQ}. Recall that a generalized quadrangle has order $(s,t)$, if every line has exactly $s+1$ points and if every point lies on exactly $t+1$ lines. We recall that $\F{x}$, for a point or a line $x$ of a generalized quadrangle, is a set of $(s+1)(t+1)$ chambers, no two of which are opposite. We call two lines of a generalized quadrangle \emph{skew} if they have no common point.

We remark that the generalized quadrangle in Example \ref{examp} is self-dual as is the set of the independent sets constructed there. Using the notation of the example, this follows easily from the fact that $Q(4,2)$ has six points that do not lie in $Q^+(3,2)$, each lying on exactly three lines, which all meet $Q^+(3,2)$ in three mutually non-collinear points.

\begin{theorem}\label{Th_GQ}
Let $X$ be a maximal set of mutually non-opposite chambers of a generalized quadrangle of order $(s,t)$. Then we have
\begin{enumerate}
\item $|X|\le (t+1)(s+1)$.
\item If $|X|=(t+1)(s+1)$, then either $X=\F{x}$ for a point or line $x$, or otherwise $s=t=2$ and $X$ is as in Example \ref{examp}.
\item If $X$ is a maximal set of chambers no two of which are opposite and if $|X|\not=(t+1)(s+1)$, then $|X|\le \max\{1+s+2t,1+t+2s\}$.
\end{enumerate}
\end{theorem}
\begin{proof}
For convenience we denote chambers in this proof by $(P,\ell)$ where $P$ is a point and $\ell$ a line on $P$.
We distinguish three cases.

Case 1. In this case we consider the situation when there exists a point or a line that occurs in at least two chambers of $X$.

Since the class of finite thick generalized quadrangles is self-dual and since the assertion of the theorem is self-dual, we may assume that there exists a line $\ell_0$ that occurs in two chambers $(P_0,\ell_0)$ and $(P'_0,\ell_0)$ of $X$. If $(P,\ell)$ is any chamber with $\ell\cap\ell_0=\emptyset$, then $P$ can be collinear to at most one of the two points $P_0$ and $P'_0$, hence $(P,\ell)$ is opposite to $(P_0,\ell_0)$ or $(P'_0,\ell_0)$, and thus $(P,\ell)$ is not a chamber of $X$. It follows that $\ell_0$ meets the lines of all chambers of $X$. Let $\LL$ be the set of all lines that are different from $\ell_0$ and occur in a chamber of $X$. Then every line of $\LL$ meets $\ell_0$ in a unique point.

Case 1a. We assume that some line $h\in\LL$ occurs in two chambers of $X$.

Let $Q$ be the point in which $h$ and $\ell_0$ meet. As before we see $h$ meets the line of every chamber of $X$. Since there are no triangles and since the line of every chamber of $X$ meets $h$ and $\ell_0$, then $Q$ lies in the line of every chamber of $X$. This implies that $X\subseteq\F{Q}$, so maximality of $X$ implies that $X=\F{Q}$. In this case we are done.

Case 1b. Now we consider the situation when every line $h$ of $\LL$ occurs in a unique chamber of $X$.

For $h\in\LL$ we denote by $P_h$ the point of $h$ with $(P_h,h)\in X $. If $P_h\in\ell_0$ for all $h\in\LL$, then $X\subseteq \F{\ell_0}$, so maximality of $X$ implies that $X=\F{\ell_0}$. We may therefore assume that there exists $h_1\in\LL$ with $P_{h_1}\not\in\ell_0$. Let $X_1$ be the point in which $h_1$ and $\ell_0$ meet. If the lines of $\LL$ meet $\ell_0$ either in $X_1$ or a second point of $\ell_0$, then $|\LL|\le 2t$, and hence $|X|\le s+1+2t$, since $\ell_0$ can occur in $s+1$ flags of $X$. In this case we are done. We may therefore assume that there are lines $h_2,h_3\in\LL$ such that the points $X_2=h_2\cap\ell_0$ and $X_3=h_3\cap\ell_0$ are distinct and distinct from $X_1$. Then any two of the lines $h_1$, $h_2$ and $h_3$ are skew and since no two of the chambers $(P_{h_i},h_i)$, $i=1,2,3$, are opposite, the points $P_{h_1}$, $P_{h_2}$ and $P_{h_3}$ are mutually collinear. Since a generalized quadrangle has no triangle, this implies that these three points are collinear, that is there exists a line $g$ containing these three points. As $P_{h_1}$ is not on $\ell_0$, then $g$ and $\ell_0$ are skew lines. If $h$ is any line of $\LL$, then $h$ meets $\ell_0$ and hence meets at most one of lines $h_1,h_2,h_3$, and consequently $P_h$ is collinear with at least two points of $g$, which implies $P_h\in g$. Hence $P_h\in g$ for all $h\in\LL$. Therefore each line of $\LL$ meets $g$ and $\ell_0$ and consequently $|\LL|\le s+1$. It follows $|X|\le s+1+|\LL|=2s+2$ and we are done.

Case 2. In this case we consider the situation that every point and every line occurs in at most one chamber of $X$, and that there exist flags $(P_1,\ell_1),(P_2,\ell_2)\in X$ with $P_1\in\ell_2$ and $P_2\not\in\ell_1$.

Consider $(P,\ell)\in X$. If $\ell$ is skew to $\ell_i$ for some $i\in\{1,2\}$, then $P$ and $P_i$ are collinear. Since a generalized quadrangle contains no triangles, this shows that $\ell$ can not be skew to $\ell_1$ and $\ell_2$.

Let $P'_1$ be a point of $\ell_1$ with $P'_1\not=P_1$. The hypothesis of the present case says that $(P'_1,\ell_1)\notin X$. Maximality of $X$ implies therefore that there exists a chamber $(Q,h)\in X$ that is opposite to $(P'_1,\ell_1)$ and hence $h$ is skew to $\ell_1$. Then $h$ meets $\ell_2$ in a point and since $(Q,h)$ and $(P_1,\ell_1)$ are not opposite we have that $Q=\ell_2\cap h$. Similarly every chamber of $X$ whose line is skew to $\ell_1$ has its point on $\ell_2$, and since every point of $\ell_2$ lies on at most one chamber of $X$, we see that $X$ contains at most $s-1$ chambers whose lines are skew to $\ell_1$.

Now we estimate the number of chambers of $X$ whose lines meet $\ell_1$ and which are distinct from $(P_1,\ell_1)$ and $(P_2,\ell_2)$. Let $(R,g)$ be such a chamber. Then $R\not=P_1,P_2$ and $g\not=\ell_1,\ell_2$. Assume that $P_1\in g$. Since $\ell_2$ is the line on $P_1$ that meets $h$, then $g$ and $h$ are skew. But then $(Q,h)$ and $(R,g)$ are opposite, which is impossible. Hence $g$ meets $\ell_1$ in a point $U$ other than $P_1$. Then $g$ and $\ell_2$ are skew and since $(R,g)$ and $(P_2,\ell_2)$ are not opposite, it follows that $R$ is the unique point on $g$ that is collinear to $P_2$. Then $R$ is not collinear to $Q$ and since $(Q,h)$ and $(R,g)$ are not opposite, it follows that $g$ meets $h$. It follows that $(R,g)$ is the only chamber of $X\setminus\{(P_1,\ell_1)\}$ whose line passes through $U$, namely $g$ is the line on $U$ that meets $h$ and $R$ is the point on $h$ collinear to $P_2$. Similarly, for every point of $\ell_1$ (other than $P_1$) there is at most one chamber in $X$ whose line is skew to $\ell_2$. Consequently $X$ contains at most $s$ chambers whose line is skew to $\ell_2$.

We have shown that $|X|\le 2s+1$ so we are done.

Case 3. In this case we consider the situation that for any two distinct chambers $(P_1,\ell_1),(P_2,\ell_2)\in X$ we have that $P_1\not\in\ell_2$ and $P_2\not\in\ell_1$ (and hence $P_1\not=P_2$ and $\ell_1\not=\ell_2$).

This implies for any two distinct flags of $X$ that either their lines are skew and their points are collinear or otherwise that their lines meet and their points are not collinear (that is, they are distinct from the intersection point of the two lines).

Consider $(P,\ell)\in X$, let $X_1$ be the set of chambers of $X$ whose lines are skew to $\ell$ and whose points are collinear with $P$, and let $X_2$ be the set of chambers of $X$ whose lines meet $\ell$ and whose points are not collinear with $P$. Then $X=X_1\cup X_2\cup\{(P,\ell)\}$.

Let $n_1$ be the number of lines on $P$ that contain the point of a chamber of $X_1$, and dually let $n_2$ be the number of points of $\ell_1$ that lie on the line of a chamber of $X_2$. Then $n_1\le t$ and $|X_1|\le n_1s$ and, dually, $n_2\le s$ and $|X_2|\le n_2t$.

We now improve this bound in the situation $n_2\ge 3$. In fact, if $n_2\ge 3$, then let $(Q_i,h_i)$, $i=1,2,3$, be in $X_2$ such that the three points $h_i\cap\ell$ are distinct. Then the lines $h_i$ are mutually skew, so the points $Q_i$ are mutually collinear and hence lie on a common line $g$. Clearly $g$ is skew to $\ell$. In this case for any $(Q,h)\in X_2$ we have that $h$ is skew to at least two of the lines $h_i$, hence $Q$ is collinear with at least two of the points $Q_i$ and thus $Q\in g$ and $h$ is the unique line on $Q$ that meets $\ell$. Since every point of $g$ lies on a unique line that meets $\ell$ and for one point of $g$ the intersection point with $\ell$ is the point $P$, we see that $|X_2|\le s$ in this case.

Hence we have $|X_2|\le 2t$ or $|X_2|\le s$. If $X_1=\emptyset$, then $|X|=1+|X_2|$ and we are done. We may thus assume that $X_1\not=\emptyset$.

We now improve the bound $|X_2|\le 2t$ when $n_2=2$. Assume for this that $n_2=2$, so there are two points $H_1,H_2$ on $\ell$ distinct from $P$ such that the line of every chamber of $X_2$ meets $\ell$ in $H_1$ or $H_2$. Also, both points $H_1$ and $H_2$ lie on the line of at least on chamber of $X_2$. For $i\in\{1,2\}$ let $c_i$ be the number of chambers in $X_2$ whose lines pass through $H_i$ and let these chambers be $(P_{ij},h_{ij})$ with $j=1,\dots,c_i$. Each line $h_{1j}$ is skew to each line $h_{2k}$ and therefore each point $P_{1j}$ is collinear with each point $P_{2k}$. Consider $(R,g)\in X_1$. Then for each $i\in\{1,2\}$ the line $g$ meets at most one line $h_{ij}$ and hence $R$ is collinear to at least $c_i-1$ of the points $P_{ij}$. If $R$ is collinear with $P_{1j}$ and $P_{2k}$, then these three points are pairwise collinear and hence there exists a line containing these three points. It follows that it is impossible that $c_1\ge 3$ and $c_2\ge 2$ or that $c_1\ge 2$ and $c_2\ge 3$. In fact, if $c_1$ were at least $3$ and $c_2$ were at least two, then $R$ would be collinear with at least two points $P_{1j}$ and $P_{1j'}$ and at least one point $P_{2k}$, so the three points $R,P_{1j},P_{2k}$ as well as the three points $R,P_{1j'},P_{2k}$ were collinear, which is impossible as $P_{1j}$ and $P_{1j'}$ are not collinear. Therefore either $c_1=c_2=2$ or $c_1=1$ or $c_2=1$. As $c_i\le t$, this implies that $|X_2|=4$ or $|X_2|\le t+1$.

Reviewing the bounds on $|X_2|$ for $n_2$ equal to 1, 2 or at least 3, we have $|X_2|\le\max\{t+1,4,s\}$. Dually we have $|X_1|\le\max\{s+1,4,t\}$. It follows that $|X|=1+|X_1|+|X_2|\le \max\{2s+t+1,2t+s+1\}$ except when $|X_1|=|X_2|=4$ and $s=t=2$.

Suppose finally that $|X_1|=|X_2|=4$ and $s=t=2$. Then $|X|=9=(s+1)(t+1)$. It follows that the four points that are collinear to $P$ but not on $\ell$ all occur as points of chambers of $X$, whereas the two points of $\ell$ other than $P$ do not occur in chambers of $X$. As $(P,\ell)$ was any chamber of $X$, the same holds for all chambers of $X$. Hence, if $T$ is the set of the nine points that occur in chambers of $X$, then each point of $T$ lies on exactly two lines all of whose points are contained in $T$. Since $|T|=|X|=9$, it follows that the points of $T$ must be the points of the six lines of a 3 by 3 grid. Since the six lines have more than one point in $T$, these lines do not occur in chambers of $X$, and hence $X$ consists of the nine chambers $(P,\ell)$ with $P\in T$ and $\ell$ the line on $P$ that is not in the grid. As the generalized quadrangle has order $(2,2)$ and is thus the parabolic quadrangle related to $Q(4,2)$, any 3 by 3 grid is an embedded hyperbolic quadric $Q^+(3,2)$. This shows that $X$ is as described in Example \ref{examp}.
\end{proof}

\textsc{Remark}:
There exist generalized quadrangles for which the bound in Theorem \ref{Th_GQ} (c) is sharp. To see this consider the parabolic quadrangle of order $(s,t)=(q,q)$, an embedded hyperbolic quadrangle of order $(q,1)$, let $Q$ be one of the points and $h,h'$ the two lines of the hyperbolic quadrangle on this point. Let $X$ be the set of all chambers $(P,\ell)$ where either $P=Q$ or otherwise $P$ is a point other than $Q$ on $h$ or $h'$ and $\ell$ is the line of the hyperbolic quadric on $P$ that is distinct from $h$ and $h'$. Then no two chambers of $X$ are opposite and $|X|=t+1+2s$.

We recall that ovoids and spreads of generalized quadrangles $G$ have been defined in the introduction. If $G$ has finite order $(s,t)$, then it has $(s+1)(st+1)$ points and since every line has $s+1$ points, it follows that spreads of $G$ have $st+1$ lines. Dually, an ovoid of $G$ consists of $st+1$ points.

\begin{corollary}
Let $G$ be a finite thick generalized quadrangle of order $(s,t)$ and let $\Gamma$ be the graph whose vertices are the chambers of $G$ with two vertices adjacent if and only if the corresponding chambers are opposite. Then we have.
\begin{enumerate}
\renewcommand{\labelenumi}{\rm (\alph{enumi})}
\item The independence number of $\Gamma$ is $(s+1)(t+1)$.
\item The chromatic number of $\Gamma$ is at least $st+1$.
\item The chromatic number of $\Gamma$ is $st+1$ if and only if $G$ possesses a spread or an ovoid.
\end{enumerate}
\end{corollary}
\begin{proof}
(a) Theorem \ref{Th_GQ} shows that the independence number of $\Gamma$ is at most $(s+1)(t+1)$. Since the set $\F{P}$ for a point $P$ is an independent set with $(s+1)(t+1)$ elements, this proves (a).

(b) The number of points of $G$ is $(s+1)(st+1)$ and hence the number of chambers of $G$ is $(s+1)(t+1)(st+1)$, that is $\Gamma$ has this many vertices. Therefore (a) implies that the chromatic number of $\Gamma$ is at least $st+1$ with equality if and only if $\Gamma$ has $st+1$ mutually disjoint independent sets of cardinality $(s+1)(t+1)$.

If $G$ has an ovoid $\cal O$, then the independent sets $\F{P}$ with $P\in\cal O$ are mutually disjoint and hence the chromatic number is $st+1$. Dually, if $G$ has a spread, then the chromatic number is $st+1$.

Conversely assume that the chromatic number of $\Gamma$ is $st+1$, that is $\Gamma$ has mutually disjoint independent sets $F_i$, $i=1,\dots,st+1$, of cardinality $(s+1)(t+1)$. We have to show that $G$ has a spread or an ovoid. As the unique generalized quadrangle of order $(2,2)$ has an ovoid, we may assume that $(s,t)\not=(2,2)$. Then Theorem \ref{Th_GQ} shows for $1\le i\le st+1$ that there exists a point or line $x_i$ with $F_i=\F{x_i}$. However, for any choice of a point $P$ and a line $\ell$, there exists a line $h$ on $P$ that contains a point $Q$ of $\ell$, and hence $(Q,h)\in\F{P}\cap\F{\ell}$. Since the independent sets $F_i$ are mutually disjoint, this shows that the $x_i$ are all points or they are all lines. If they are all points, then these points are mutually non-collinear, since the sets $\F{x_i}$ are disjoint. In this case, the set $T:=\{x_i\mid 1\le i\le st+1\}$ is an ovoid. Dually, if all $x_i$ are lines, then $T$ is a spread.
\end{proof}

\textsc{Remark}. For most known finite generalized quadrangles it is known that they either have a spread or an ovoid, for information we refer to \cite{Payne&Thas}. This is true for all families of classical generalized quadrangles related to finite polar spaces of rank two with one exception. For the generalized quadrangle of order $(q^2,q^3)$ consisting of the points and lines of a hermitian polar space $H(4,q^2)$ it is known that this generalized quadrangle has no ovoid. Brouwer \cite{Brouwer} showed that it has no spread when $q=2$, so the related graph whose vertices are the chambers of $H(4,4)$ as above has chromatic number $\chi$ at least $2+ q^2\cdot q^3=34$. We mention that $\chi\le 36$, since given a point $p$ of $H(4,4)$, then the union of the $36$ independent sets $\F{x}$ with $x\not=p$ and $x$ collinear with $p$ contains all chambers of $H(4,4)$. Hence $\chi\in\{34,35,36\}$.  For $q>2$ the existence of spreads of $H(4,q^2)$ is open and therefore the chromatic number of the related graph remains open.

\section{A stability result}

In this section we prove parts (a) and (b) of Theorem \ref{mainPG}. Therefore we consider chambers of the finite projective $3$-space $\PG(3,q)$ of order $q$ and we are interested in large sets of chambers that are mutually non-opposite. As already mentioned in the introduction, it follows from \cite{Sam} that such a set has at most $(q^2+q+1)(q+1)^2$ elements. Moreover, for each point or plane $x$ of $\PG(3,q)$, the set $\F{x}$ consisting of all chambers whose line is incident with $x$ is such a set and has exactly $(q^2+q+1)(q+1)^2$ elements. We shall see for $q\ge 43$ that these are the only examples meeting the bound.

\begin{theorem}\label{stabilityA3}
Let $X$ be a maximal family of pairwise non-opposite chambers of $PG(3,q)$, $q\ge 43$. Then either $X=\F{x}$ for a point or plane $x$, or $|X|\le 9(q+1)(5q^2+1)$.
\end{theorem}

The proof will be prepared in several lemmata. Throughout this section we assume that $X$ is a maximal set of mutually non-opposite chambers of $\PG(3,q)$. For simplicity, we denote a chamber in this section as an ordered triple $(P,\ell,\pi)$ with a point $P$, a line $\ell$ and a plane $\pi$ of $\PG(3,q)$ that are pairwise incident.

\begin{lemma}\label{EigenschaftenPunkt}
Let $R$ be a point or a plane with $X\not=\F{R}$. Then $|X\cap\F{R}|\le (q+1)^3+q^2(2q+1)$.
\end{lemma}
\begin{proof}
As $X$ is maximal, there exists a chamber $(Q,h,\tau)$ in $X\setminus \F{R}$. Then $R$ is not incident with $h$ and hence $R$ is incident with $q^2$ lines $\ell$ that are skew to $h$. Each such line $\ell$ lies in $q$ planes $\pi$ that do not contain the point $Q$, and each such line $\ell$ contains $q$ points $P$ that do not lie in $\tau$. This gives $q^2\cdot q\cdot q=q^4$ chambers $(P,\ell,\pi)$ of $\F{R}$ that are opposite to $(Q,h,\tau)$ and hence do not lie in $X$. Therefore $|X\cap\F{R}|\le (q^2+q+1)(q+1)^2-q^4$.
\end{proof}

\begin{lemma}\label{incidentmitGerade}
Let $\ell$ be a line.
\begin{enumerate}\renewcommand{\labelenumi}{\rm (\alph{enumi})}

\item At most $(q+1)^3$ chambers $(Q,h,\tau)$ of $X$ satisfy $Q\in \ell\subseteq \tau$.

\item At most $(q+1)q^2$ chambers $(Q,h,\tau)$ of $X$ satisfy $Q\notin\ell$ and $\ell\subseteq \tau$.

\item At most $(q+1)q^2$ chambers $(Q,h,\tau)$ of $X$ satisfy $Q\in\ell$ and $\ell\not\subseteq \tau$.
\end{enumerate}
\end{lemma}
\begin{proof}
(a) There are only $(q+1)^3$ chambers $(Q,h,\tau)$ with $Q\in\ell\subseteq \tau$, so at most this many in $X$.

(b) Consider the chambers $(Q,h,\tau)\in X$ with $Q\notin\ell$ and $\ell\subseteq \tau$. For any two such chambers, their two lines $h$ must meet (this is clear if the two chambers have the same plane, and otherwise it follows from the fact that the two chambers are not in general position). Hence all these lines lie in the same plane $\pi$ on $\ell$ or pass through the same point $P$ of $\ell$. In either case, there are at most $q^2+q$ different lines $h$ and thus at most $(q^2+q)q$ such chambers.

(c) This is the statement dual to the one in (b).
\end{proof}

\begin{lemma}\label{wievieletreffenl0Teil1}
Let $\ell_0$ be a line and let $\T$ be the set consisting of all chambers $(Q,h,\tau)$ of $X$ that satisfy $h\cap \ell_0\not=\emptyset$, $Q\notin \ell_0$ and $\ell_0\not\subseteq \tau$. Suppose that $q\ge 5$ and that $X\not=\F{x}$ for every point and every plane $x$. Then
$|\T|\le (q+1)(8q-6)q$.
\end{lemma}
\begin{proof}
For each $P\in\ell_0$ let us denote by $\T_p$ the subset of $\T$ that consists of those chambers of $\T$ whose line contains $P$. Then $\T$ is the disjoint union of the sets $\T_P$ with $P\in\ell_0$. We distinguish two cases.

Case 1. We assume for every point $P$ of $\ell_0$ that there exist two chambers $(Q_i,h_i,\tau_i)\in\T$, $i=1,2$, with $h_1\cap h_2=\emptyset$ and $P\notin h_1,h_2$.

Let $P$ be a point of $\ell_0$ and let $(Q_1,h_1,\tau_1)$ and $(Q_2,h_2,\tau_2)$ be chambers of $\T$ with $h_1\cap h_2=\emptyset$ and $P\notin h_1,h_2$. Since $h_1$ and $h_2$ are skew and both meet $\ell_0$, we see that $h_1$ and $h_2$ span distinct planes with $\ell_0$. We first determine an upper bound for the number of chambers $(Q,h,\tau)$ in $\T_P$ for which $h\cap h_2\not=\emptyset$. For such a chamber we have $h\cap h_1=\emptyset$ and therefore either $Q_1\in \tau$ or $Q\in \tau_1$. If $Q_1\in \tau$, then $\tau$ is the plane spanned by $h$ and $Q_1$, and $Q$ can be any of the $q$ points of $h$ that are different from $P$. If $Q_1\notin \tau$, then $Q\in \tau_1$ and hence $Q$ is the intersection point of $h$ and $\tau_1$. In this case $\tau$ can be any of the $q-1$ planes on $h$ that contain neither $\ell_0$ nor $Q_1$. Since $P$ lies on $q$ lines other than $\ell_0$ that meet $h_2$, it follows that at most $q(2q-1)$ chambers of $\T_P$ have a line that meets $h_2$. The same argument shows that $\T_P$ contains at most $q(2q-1)$ chambers whose line meets $h_1$.

Now we consider only the chambers $(Q,h,\tau)\in \T_P$ for which $h$ misses $h_1$ and $h_2$. Since $(Q,h,\tau)$ is not in general position to $(Q_i,h_i,\tau_i)$ we must have $Q_i\in \tau$ or $Q\in \tau_i$ for $i=1$ and $i=2$.

If $Q_1,Q_2\in \tau$, then $h$ meets the line $Q_1Q_2$, and $\tau$ is the plane spanned by the lines $h$ and $Q_1Q_2$. There are at most $(q-1)q$ chambers in $\T_P$ with this property, since there are only $q-1$ lines $h$ on $P$ that meet $Q_1Q_2$ and miss $h_1$ and $h_2$, and for each line $h$ there are $q$ choices for $Q$. Dually, there are at most $(q-1)q$ chambers in $\T_P$ whose point $Q$ lies in $\tau_1\cap \tau_2$. If $Q_1\in \tau$ and $Q\in \tau_2$, then $Q=h\cap \tau_2$ and $\tau$ is the plane spanned by $h$ and $Q_1$. There are at most $q^2-q$ such chambers in $\T_P$, since this is the number of lines $h$ on $P$ that miss $h_1$ and $h_2$. The same argument gives the bound $q^2-q$ for the number of chambers $(Q,h,\tau)$ under consideration with $Q_2\in \tau$ and $Q\in \tau_1$.

We have shown that $|\T_P|\le 2q(2q-1)+4q(q-1)=(8q-6)q$ for all $P\in\ell$. Hence $|\T|\le (q+1)(8q-6)q$.

Case 2. We assume that there exists a point $P\in\ell_0$ that does not have the property of Case 1. Hence, if we consider all chambers $(Q,h,\tau)\in \T$ with $P\notin h$, then the lines of these chambers mutually meet. Consequently there exists a plane or a point that is incident with all these lines. Therefore Lemma \ref{EigenschaftenPunkt} shows that $|\T|-|\T_P|\le (q+1)^3+q^2(2q+1)$. The same lemma shows that $|\T_P|\le (q+1)^3+q^2(2q+1)$. Hence $|\T|\le 2(q+1)^3+2q^2(2q+1)$. Since $q\ge 5$, this bound is better than the one in the statement.
\end{proof}

\begin{lemma}\label{wievieletreffenl0}
Suppose that $q\ge 5$ and $X\not=\F{x}$ for every point and every plane $x$. Then for every line $\ell$ there exist at most $c:=11q^3+7q^2-3q+1$ chambers in $X$ whose lines meet $\ell$.
\end{lemma}
\begin{proof}
Let $\ell$ be a line. Lemma \ref{incidentmitGerade} shows that at most $2(q+1)q^2+(q+1)^3$ chambers of $X$ have their point on $\ell$ or their plane through $\ell$. This number is equal to $3q^3+5q^2+3q+1$. Therefore Lemma \ref{wievieletreffenl0Teil1} proves the statement.
\end{proof}

\begin{lemma}\label{twoskewlines}
If $q\ge 5$ and if $X$ contains two chambers whose lines are skew, then $|X|\le 9(q+1)(5q^2+1)$.
\end{lemma}
\begin{proof}
Suppose that $q\ge 5$ and $X$ has two chambers $(P_1,\ell_1,\pi_1)$ and $(P_2,\ell_2,\pi_2)$ with $\ell_1\cap\ell_2=\emptyset$. Then $X\not=\F{y}$ for every point and every plane $y$. Defining the number $c$ as in Lemma \ref{wievieletreffenl0}, it follows that every line meets the line of at most $c$ chambers of $X$. In view of the bound that we want to show for $|X|$, we may assume that $|X|>2c$. Then there exists a chamber $(P_3,\ell_3,\pi_3)$ in $X$ such that $\ell_3$ is skew to $\ell_1$ and to $\ell_2$. Lemma \ref{wievieletreffenl0} shows that $X$ contains at most $3c$ chambers whose lines meets at least one of the lines $\ell_1$, $\ell_2$, and $\ell_3$. Now we will determine an upper bound for the number of chambers of $X$ whose lines are skew to $\ell_1$, $\ell_2$ and $\ell_3$.

Let $(P,\ell,\pi)$ be such a chamber of $X$. For each $i\in\{1,2,3\}$ we have that $Q\in\pi_i$ or $Q_i\in \pi$. Hence $\pi$ contains at least two of the points $Q_i$, or $Q$ lies in at least two of the planes $\pi_i$. Notice that the fact that the lines $\ell_1$, $\ell_2$, and $\ell_3$ are pairwise skew implies that the points $Q_1$, $Q_2$, $Q_3$ are pairwise distinct and that the planes $\pi_1$, $\pi_2$, and $\pi_3$ are pairwise distinct.

The number of chambers of $X$ whose planes contain two given points of $\{Q_1,Q_2,Q_3\}$ is at most $d:=(q+1)q^2+(q+1)^3$ by Lemma \ref{incidentmitGerade}. Hence there are at most $3d$ chambers in $X$ whose plane contains two of the points $Q_i$. Dually, there exist at most $3d$ chambers in $X$ whose point is contained in two of the planes $\pi_i$.

Therefore $|X|\le 3c+6d=9(q+1)(5q^2+1)$. This proves the lemma.
\end{proof}

\begin{proof}[Proof of Theorem \ref{stabilityA3}] Let $X$ be a maximal family of pairwise non-opposite chambers of $PG(3,q)$ with $q\ge 43$. If $X$ contains two chambers, whose lines are skew, then Lemma \ref{twoskewlines}  shows that $|X|\le 9(q+1)(5q^2+1)$. If the lines of the chambers of $X$ mutually meet, then they all pass through a common point $P$ or all lie in a common plane $\pi$. In this case we have $X\subseteq\F{P}$ or $X\subseteq\F{\pi}$ and maximality of $X$ shows that $X=\F{P}$ or $X=\F{\pi}$.
\end{proof}

\begin{corollary}\label{stability}
If $q\ge 43$, then every maximal set of chambers of $\PG(3,q)$ that are mutually not opposite has either $(q^2+q+1)(q+1)^2$ elements or at most $9(q+1)(5q^2+1)$ elements.
\end{corollary}

\textsc{Remark}. With some effort one can improve the bound in Lemma \ref{wievieletreffenl0} but I was not able to prove the desired result for all $q$, so I omit these better bounds.

These results become becomes in a way more natural when translated by the Klein-correspondence to the hyperbolic quadric $Q^+(5,q)$. Here a chamber is a triple $(P,\pi,\tau)$ consisting of a point $P$, a Greek plane $\pi$ and a Latin plane $\tau$ such that $P$ lies in the planes $\pi$ and $\tau$. Two chambers $(P_i,\pi_i,\tau_i)$ are opposite, if the points $P_1$ and $P_2$ are non-collinear and if $\pi_1\cap \tau_2=\pi_2\cap \tau_1=\emptyset$. Since points and planes of projective 3-space translate by Klein-correspondence both to planes, the results of Theorem \ref{stabilityA3} and Corollary \ref{stability} translate into the following slightly more pleasant formulation.

\begin{corollary}
If $q\ge 43$, then a set of chambers of $Q^+(5,q)$ that are mutually not opposite has at most $(q^2+q+1)(q+1)^2$ elements and equality occurs if and only if the set consists of all chambers whose point is contained in a given plane.
\end{corollary}

\section{The chromatic number}

In this section we prove part (c) of Theorem \ref{mainPG} and for this we use the following notation. Let $\Gamma$ be the graph whose vertices are the chambers of $\PG(3,q)$ with two vertices being adjacent when the corresponding chambers are opposite. Let $V$ be the vertex set of $\Gamma$.

First we give an example to show that the chromatic number is at most $q^2+q$.

\begin{example}
Let $\ell$ be a line of $\PG(3,q)$, let $\pi$ be a plane on $\ell$, and let $C$ be the set consisting of the $q^2$ points of $\pi$ that do not lie on $\ell$ and of the $q$ planes on $\ell$ that are distinct from $\pi$. Then every line of $\PG(3,q)$ is incident with at least one element of $C$. Hence every chamber of $\PG(3,q)$ lies in $\F{x}$ for at least one element $x$ of $C$. This shows that the vertex set of $\Gamma$ can be covered with $|C|=q^2+q$ independent sets of $\Gamma$. This implies that the chromatic number of $\Gamma$ is at most $q^2+q$.
\end{example}

\begin{theorem}\label{ThPGaundb}
For $q\ge 47$ the chromatic number of the graph $\Gamma$ defined in the beginning of this section is $q^2+q$.
\end{theorem}
\begin{proof}
Let $\chi$ be the chromatic number of $\Gamma$. Then there exists independent sets $G_1,\dots,G_\chi$ that partition the vertex set $V$ of $\Gamma$. From the above example we know that $\chi\le q^2+q$. For $i=1,\dots,\chi$ let $F_i$ be a maximal independent set of $\Gamma$ with $G_i\subseteq F_i$. Put $\theta_2=q^2+q+1$, $e_0=\theta_2(q+1)^2$ and $e_1=9(q+1)(5q^2+1)$. From Corollary \ref{stability} we know for each $i$ that $|F_i|=e_0$ or $|F_i|\le e_1$.

Let $I$ be the set of indices $i$ with $|F_i|=e_0$. For $i\in I$, Theorem \ref{stabilityA3} shows that there exists a point or plane $x_i$ such that $F_i=\F{x_i}$. Put $A=\{x_i\mid i\in I,\ \text{$x_i$ is a point}\}$ and $B=\{x_i\mid i\in I,\ \text{$x_i$ is a plane}\}$. It follows from Theorem 2.2 of \cite{MetschHowMany} that the number of lines of $\PG(3,q)$ that contain a point of $A$ is at most $\theta_2+|A|q^2$. Dually the number of lines of $\PG(3,q)$ that are contained in a plane of $B$ is at most $\theta_2+|B|q^2$. Since the line of every chamber of $\F{x}$ with $x\in A$ contains $x$, it follows that at most $(\theta_2+|A|q^2)(q+1)^2$ chambers are contained in $\F{x}$ for some $x\in A$, and similarly at most $(\theta_2+|B|q^2)(q+1)^2$ chambers are contained in $\F{x}$ for some $x\in B$. Since the total number of chambers is $(q^2+1)\theta_2(q+1)^2$ and since every chamber lies in at least one set $F_i$, $1\le i\le\chi$, follows that
\begin{align}\label{chischranke}
(q^2+1)\theta_2(q+1)^2&\le \left|\bigcup_{i=1}^\chi F_i\right|\le (2\theta_2+|A|q^2+|B|q^2)(q+1)^2+(\chi-|A|-|B|)e_1.
\end{align}
Since $q\ge 47$, then $q^2(q+1)^2>e_1$. Since $|A|+|B|=|I|\le\chi$, it follows that
\begin{align*}
(q^2+1)\theta_2(q+1)^2&\le (2\theta_2+\chi q^2)(q+1)^2.
\end{align*}
This implies that $\chi>q^2+q-1$ and hence $\chi=q^2+q$.
\end{proof}

\textsc{Remark}. The proof shows slightly more for $q\ge 47$. In fact using $\chi=q^2+q$ it follows from \eqref{chischranke} that $|A|+|B|>q^2+q-1$. Therefore every coloring with $\chi=q^2+q$ colors comes from a covering of the set of chambers with $q^2+q$ sets $\F{x_i}$, $i=1,\dots,q^2+q$, where each $x_i$ is a point or a plane. It follows that each line of $\PG(3,q)$ is incident with one element $x_i$.
If $L(x_i)$ is the set of lines incident with $x_i$, it follows that the sets $L(x_i)$ provide a covering of the Kneser graph of lines by independent sets. These Kneser graphs have been studied in \cite{KneserLines} where it was shown that they also have chromatic number $q^2+q$. Clearly every coloring of the Kneser graph of lines of $\PG(3,q)$ provides a coloring of the graph of chambers studied in the present paper, by replacing each color class $C$ by the set of all chambers whose line lies in $C$.
Therefore the optimal colorings in both graphs correspond to one another and therefore we refer to \cite{KneserLines} for further information.

\bibliographystyle{plain}

\end{document}